\documentclass[12pt]{amsart}
\usepackage{url}
\usepackage{amsmath,amsxtra,amssymb,latexsym,epsfig,amscd,amsthm,fancybox,epsfig}
\usepackage[mathscr]{eucal}
\usepackage{graphicx}
\usepackage{multicol,xcolor}
\usepackage{epsfig} 
\usepackage{epstopdf}
\usepackage{cases}
\usepackage{subfig}
\usepackage{color}
\usepackage{hyperref, enumerate}
\usepackage{bigints}
\usepackage{dsfont}
\usepackage{thmtools,thm-restate}

\newcommand{\1}{\mathds{1}}

\newcommand{\diff}[3][]{\frac{d^{#1} #2}{d #3^{#1}}}

\newcommand{\comment}[1]{}

\numberwithin{equation}{section}

\def\bb{\begin{equation}}
\def\ee{\end{equation}}

\def\bse{\begin{subequations}}
\def\ese{\end{subequations}}

\newcommand{\F}{\mathcal{F}}

\DeclareMathOperator{\sign}{sgn}
\DeclareMathOperator{\supp}{supp}

\def\R{\mathbb{R}}

\def\Pr{\mathbb{P}}
\def\E{\mathbb{E}}

\def\supp{\text{supp}}
\def\pa{\partial}

\newtheorem{defin}{Definition}[section]
\newtheorem{lemma}[defin]{Lemma}
\newtheorem{prop}[defin]{Proposition}
\newtheorem{theorem}[defin]{Theorem}
\newtheorem{corr}[defin]{Corollary}
\newtheorem{claim}[defin]{Claim}
\newtheorem{assu}[defin]{Assumption}
\newtheorem{rema}[defin]{Remark}
\newtheorem{conj}[defin]{Conjecture}

\title[The Inverse First Passage Time Problem]{The Inverse First Passage Time Problem for killed Brownian motion}
\author[B. Ettinger]{Boris Ettinger}
\email{boris.ettinger@gmail.com}
\author[A. Hening]{Alexandru Hening }
\address{Department of Mathematics\\
Tufts University\\
Bromfield-Pearson Hall\\
503 Boston Avenue\\
Medford, MA 02155\\
United States
}
\email{Alexandru.Hening@tufts.edu}

 \author[T. K. Wong]{Tak Kwong Wong}
\address{Department of Mathematics \\
 The University of Hong Kong\\
 Pokfulam\\
 Hong Kong}
 \email{takkwong@maths.hku.hk}


\begin{document}

\begin{abstract}
The classical \textit{inverse first passage time problem} asks whether, for a Brownian motion $(B_t)_{t\geq 0}$ and a positive random variable $\xi$, there exists a barrier  $b:\mathbb{R}_+\to\mathbb{R}$ such that $\Pr\{B_s>b(s), 0\leq s \leq t\}=\Pr\{\xi>t\}$, for all $t\geq 0$. We study a variant of the inverse first passage time problem for killed Brownian motion. We show that if $\lambda>0$ is a killing rate parameter and $\1_{(-\infty,0]}$ is the indicator of the set $(-\infty,0]$ then, under certain compatibility assumptions, there exists a unique continuous function $b:\mathbb{R}_+\to\mathbb{R}$ such that $\E\left[-\lambda \int_0^t \1_{(-\infty,0]}(B_s-b(s))\,ds\right] = \Pr\{\zeta>t\}$ holds for all $t\geq 0$. This is a significant improvement of a result of the first two authors (Annals of Applied Probability 24(1):1--33, 2014).

The main difficulty arises because $\1_{(-\infty,0]}$ is discontinuous. We associate a semi-linear parabolic partial differential equation (PDE) coupled with an integral constraint to this version of the inverse first passage time problem. We prove the existence and uniqueness of weak solutions to this constrained PDE system. In addition, we use the recent Feynman-Kac representation results of Glau (Finance and Stochastics 20(4):1021--1059, 2016) to prove that the weak solutions give the correct probabilistic interpretation.
\end{abstract}
\maketitle
\tableofcontents

\section{Introduction}\label{s:Introduction}

Suppose  $(B_t)_{t\geq 0}$ is a standard Brownian motion on a probability space $(\Omega,\F,\{\F_t\}_{t\geq0},\Pr)$ with a filtration $\{\F_t\}_{t\geq 0}$ satisfying the usual conditions. For any Borel measurable function $b:\R_+\rightarrow \R$, we define the stopping time
\bb\label{e:hat_tau}
\hat \tau:= \inf\left\{t>0: B_t\leq b(t)\right\}.
\ee
This is the first time the Brownian motion $(B_t)_{t\geq 0}$ goes below the barrier $b$. There are two interesting classical problems involving first passage times.
\begin{itemize}
\item The classical \textit{first passage time problem (FPT)}: For a given function $b:\R_+\to\R$, find the survival distribution of the first time that $(B_t)_{t\geq 0}$ crosses $b$. In other words, find
\[
\Pr\{\hat \tau>t\},\quad\mbox{for all } t\geq 0.
\]
\item The classical \textit{inverse first passage time problem (IFPT)}: For a given survival function $G:[0,\infty)\to[0,\infty)$, does there exist a Borel measurable function $b:\R_+\rightarrow \R$ such that
\[
G(t) = \Pr\{\hat \tau>t\}
\]
for all $t\geq 0$?
\end{itemize}
First passage times of Markov processes through constant or time-dependent barriers have been used extensively to model phenomena in finance (\cite{HW00, HW01, AZ01}), neuroscience (\cite{BSZ13}), meteorology,  engineering, and biology -- see \cite{RS90, EEH14, ZS09, DP15} for further references.

There have been numerous papers looking at the first passage time problem (\cite{K52, DS53, P02, W07}). Due to applications in mathematical finance, the inverse first passage time problem has seen increased interest in recent years. The aim of this paper is to solve the first passage and inverse first passage time problems for killed diffusions.

The classical inverse first passage time problem has originally been posed by A. Shiryaev in 1976 for the special case when the distribution of the first passage time is exponential. A first step towards solving the problem was taken in \cite{A81} where the author showed that there exists a stopping time with the given distribution. Nevertheless, this stopping time is not the first hitting time of a barrier $b$ by a Brownian motion $(B_t)_{t\geq 0}$.

A large class of first passage time problems can be analyzed using a partial differential equations (PDE) framework. Let $w(t,x):= \frac{\partial}{\partial x}\Pr\{B_t\leq x, \hat \tau>t\} $ be the sub-probability density of $(B_t)_{t\geq 0}$ killed at $\hat \tau$. Using the Kolmogorov forward equation, one can see that $w$ satisfies
\begin{equation} \label{e_Kolm_FPT}
\left\lbrace
 	\begin{aligned}
		\partial_t w(t,x) &= \frac{1}{2} \partial_x^2 w(t,x) - \partial_x w(t,x), \quad x>b(t),\, t>0, \\
        w(t,x) &=0, \quad x\leq b(t),\, t>0,\\
        w(0,x)&= f(x), \quad x \in \R,\\
	\end{aligned}
 \right.
\end{equation}
where the function $f$ is the probability density of $B_0$. When the function $b$ is smooth enough, \eqref{e_Kolm_FPT} has a unique smooth solution, and we can express the survival probability as
\[
G(t):= \Pr\{\hat \tau>t\} = \int_{b(t)}^\infty w(t,x)\,dx,\quad\mbox{for all }t\geq 0.
\]
An important step towards solving the IFPT was taken in \cite{AZ01} where the authors show that for sufficiently smooth barriers $b$ and survival probabilities $G$, the density $w$ and the barrier $b$ are a solution to the following free boundary problem
\begin{equation} \label{e_Kolm_IFPT}
\left\lbrace
 	\begin{aligned}
		\partial_t w(t,x) &= \frac{1}{2}\partial_x^2 w(t,x) -  \partial_x w(t,x), \quad x>b(t),\, t>0, \\
        w(t,x) &=0, \quad x\leq b(t),\, t>0,\\
        w(0,x)&= f(x), \quad x \in \R,\\
        G(t) &= \int_{b(t)}^{\infty} w(t,x)\,dx, \quad t\geq 0.
	\end{aligned}
 \right.
\end{equation}
The existence and uniqueness of viscosity solution of \eqref{e_Kolm_IFPT} was
established in \cite{CCCS06}. In the follow-up paper \cite{CCCS11} the authors show that, when $G$ is continuous, the solution $b$ of \eqref{e_Kolm_IFPT} gives the correct probabilistic interpretation and therefore solves the IFPT. The recent paper \cite{EJ16} provides a different proof for the classical IFPT problem by using an elegant connection to optimal transport. We were inspired to study these problems after reading the preprint \cite{DP10} which describes how first passage times can be used to quantify the credit risk of certain financial transactions.

Let $U$ be an exponentially distributed random variable with mean one that is independent of $(B_t)$, and let  $\1_{(-\infty,0]}$ be the indicator function of the set $(-\infty,0]$.  Assume $0\leq \psi\leq 1$ is a suitably smooth approximation of $\1_{(-\infty,0]}$ that is non-increasing with $\lim_{x\to-\infty}\psi(x)=1$ and $\lim_{x\to\infty}\psi(x)=0$. We can define the random times
\bb\label{e_tau}
 \tau:=\inf\left\{t>0: \lambda \int_0^t \1_{(-\infty,0]}(B_s-b(s))\,ds>U\right\}
\ee
and
\bb\label{e_tilde_tau}
\tilde\tau := \inf\left\{t>0: \lambda \int_0^t \psi (B_s-b(s))\,ds>U\right\},
\ee
where $\lambda>0$ is a fixed rate parameter.
The integral  $\int_0^t \1_{(-\infty,0]} (B_s-b(s))\,ds$ is the time spent by the Brownian motion $(B_t)_{t\geq 0}$, during the period $[0,t]$, under the barrier $b$. Similarly, $ \int_0^t \psi (B_s-b(s))\,ds$ is an approximation of this time.

The random time $\tau$ (and $\tilde \tau$) is a ``smoothed-out'' version of the stopping time $\hat \tau$ from \eqref{e:hat_tau}. Instead of killing $(B)_{t\geq 0}$ as soon as it hits the barrier $b$, we kill $(B)_{t\geq 0}$ at rate $\lambda$ if it is in a state $y$ (and $y\leq b(t)$) at time $t\geq 0$. Note that if we let $\lambda\to\infty$ in \eqref{e_tau} or \eqref{e_tilde_tau} we recover the time $\hat \tau$.
\begin{rema}
We have the following possible financial interpretation if one assumes $(B_t)_{t\geq 0}$ models the credit index of a company. When $B_t$ is large, corresponding to a time $t$ when the counterparty is in sound financial health, the killing rate $\lambda\1_{(-\infty,0]}(B(t)-b(t))$ is 0 and default in an ensuing short period of time is unlikely, whereas the killing rate is equal to its maximum possible value, $\lambda$, when $B_t$ is low and default is more probable.
\end{rema}
It is straightforward to check that since $U$ is an exponentially distributed random variable with mean $1$ that is independent of $(B_t)$, we have
\bb\label{e_survival_tau}
\Pr\{\tau>t\} = \E \left[\exp\left(-\lambda \int_0^t \1_{(-\infty,0]}(B_s-b(s))\,ds\right)\right],\quad\mbox{for all } t\geq 0.
\ee
\begin{assu}\label{a:1}
We suppose for the remainder of the paper that the Brownian motion $(B_t)_{t\geq 0}$ has a random starting position $B_0$. Furthermore, we suppose that $B_0$ has a density $f\in L^2(\R)$ that is supported on $\R$, i.e., $f(x)>0$ for all $x\in \R$.
\end{assu}
As a result of Assumption \ref{a:1}, equation \eqref{e_survival_tau} becomes
\bb\label{e_survival_tau_2}
\Pr\{\tau>t\} = \int_\R \E \left[\exp\left(-\lambda \int_0^t \1_{(-\infty,0]}(x+B_s-b(s))\,ds\right)\right] f(x)\,dx,\quad\mbox{for all } t\geq 0.
\ee
\begin{rema}
From now on we will assume without loss of generality that $\lambda=1$.
\end{rema}
In \cite{EEH14} the IFPT for the random time $\tilde \tau$ defined in \eqref{e_tilde_tau} was analyzed thoroughly. We note that $\tilde \tau$ is an approximation of the more natural choice of stopping time $\tau$.
It was shown in \cite[Theorem 2.1]{EEH14}
that if $(B_t)_{t\geq 0}$ is a Brownian motion with a given suitable
random initial condition $B_0$
and the survival function $G$ is twice continuously differentiable then there is a unique differentiable function $b$
such that the stopping time $\tilde \tau$ has the survival function $G$.

In the current paper we are interested in the FPT and IFPT problems for the random time $\tau$ from \eqref{e_tau}. More specifically:
\begin{itemize}
\item The \textit{First Passage Time Problem for Killed Brownian Motion (FPTK)}: For a given Borel measurable function $b:\R_+\to\R$, find the survival distribution of the time when the Brownian motion $(B_t)_{t\geq 0}$ is killed while being under the barrier $b$. That is, find for all $t\geq 0$,
\[
\Pr\{ \tau>t\}=\int_\R \E \left[\exp\left(-\lambda \int_0^t \1_{(-\infty,0]}(x+B_s-b(s))\,ds\right)\right] f(x)\,dx.
\]
\item The \textit{Inverse First Passage Time Problem for Killed Brownian Motion (IFPTK)}: For a given survival function $G:[0,\infty)\to[0,\infty)$, does there exist a function $b$ such that
\[
G(t) = \Pr\{ \tau>t\} = \int_\R \E \left[\exp\left(-\lambda \int_0^t \1_{(-\infty,0]}(x+B_s-b(s))\,ds\right)\right] f(x)\,dx
\]
for all $t\geq 0$?
\end{itemize}

\begin{rema}
It was key in the arguments from \cite{EEH14} to assume that $\psi$ was a smooth enough approximation of the indicator $\1_{(-\infty,0]}$. Theorem 4.1 from \cite{EEH14} shows that there exists a solution to the IFPT problem for $\tau$. However, it does not yield the uniqueness of the barrier function $b$ nor any regularity properties.
\end{rema}
In order to find the barrier $b$ satisfying the IFPTK we will study the following related PDE problem: for any given initial data
\begin{equation*}
u(0,x)=u_0(x) >0 \quad \mbox{for any }x\in\R \quad\text{and}\quad b(0)=b_0\in \R,
\end{equation*}
and any smooth function $G:=G(t)$ satisfying appropriate compatibility condition(s) that we will discuss below, we want to find a barrier function $b:=b(t)$ such that the unique solution $u:=u(t,x)$ to
\begin{equation}\label{e:MainSystem}
\left\{\begin{aligned}
\partial_t u &= \frac{1}{2} \partial_x^2 u - \1_{(-\infty,b(t)]} u \\
u(0,x) &= u_0(x) >0\\
b(0) &=b_0\in\R
\end{aligned}\right.
\end{equation}
satisfies the following identity
\begin{equation}\label{e:MassIdentity}
G(t)=\int^\infty_{-\infty} u(t,x) \;dx, \quad\mbox{for all }t\geq 0.
\end{equation}
Here, $\1_{(-\infty,b(t)]}:=\1_{(-\infty,b(t)]} (x)=\1_{(-\infty,0]} (x-b(t))$ is the indicator function of the set $(-\infty,b(t)]$. We assume that $G\in C^1$. Here we consider that $u_0$, $b_0$ and $G$ are given, and $u$ and $b$ are the unknowns. We will study the existence and uniqueness of solutions $(u,b)$ for the system \eqref{e:MainSystem} with the constraint \eqref{e:MassIdentity}.

\begin{rema}\label{r:classical}
We note that \eqref{e:MainSystem} cannot have smooth classical solutions. If one assumed that $b\in C(\R)$ and $u\in C^{1,2}([0,T]\times \R)$ for some $T>0$ then one could obtain from \eqref{e:MainSystem} that for any $t\in (0,T)$, the function
\[
x\mapsto \1_{(-\infty,0]}(x- b(t)) u(t,x)
\]
is continuous; something which is clearly false. As such one needs to work with suitable weak solutions.
\end{rema}

The \textit{hazard rate} of the random time $\tau$ is given by
\[
\frac{\Pr\{\tau \in dt ~|~ \tau > t\}}{dt}.
\]
The following heuristic shows that, due to the specific form \eqref{e_tau} of $\tau$, there will be restrictions on the hazard rates that can be covered by our model.
A straightforward computation yields
\begin{eqnarray}\label{e_hazard}
\frac{\Pr\{\tau \in dt ~|~ \tau > t\}}{dt}&:=&\lim_{\Delta t\downarrow 0} \frac{\Pr\{\tau\in(t,t+\Delta t)\}}{\Delta t\Pr\{\tau > t\}}\nonumber\\
&=& \lim_{\Delta t\downarrow 0} \frac{\Pr\left\{\int_0^t \1_{(-\infty,0]}(B_s-b(s))ds\leq U\leq \int_0^{t+\Delta t}\1_{(-\infty,0]}(B_s-b(s))\,ds \right\}}{\Delta t\Pr\left\{ \int_0^t \1_{(-\infty,0]}(B_s-b(s))\,ds\leq U \right\}}\\
&=& \lim_{\Delta t\downarrow 0} \frac{\E\left[e^{-\int_0^t \1_{(-\infty,0]}(B_s-b(s))\, ds}-e^{- \int_0^{t+\Delta t}\1_{(-\infty,0]}(B_s-b(s))\,ds }\right]}{\Delta t\E\left[\exp\left(- \int_0^t \1_{(-\infty,0]}(B_s-b(s)) \,
ds\right)\right]}\nonumber\\
&=&\frac{  \E\left[\1_{(-\infty,0]}(B_t-b(t)) \exp\left(- \int_0^t \1_{(-\infty,0]}(B_s-b(s))\, ds\right)\right]}{\E\left[\exp\left(- \int_0^t \1_{(-\infty,0]}(B_s-b(s))\, ds\right)\right]}.\nonumber
\end{eqnarray}
On the other hand, suppose that $\zeta$ is a non-negative random variable with
survival function $t \mapsto G(t):= \mathbb{P}\{\zeta > t\}$.
The corresponding hazard rate is
\[
-\frac{G'(t)}{G(t)} = - \frac{d}{dt} \log G(t).
\]
As a result of \eqref{e_hazard}, a necessary condition for a function $b$ to exist such that
the corresponding random time $\tau$ has the same distribution as $\zeta$
is that
\begin{equation}
\label{e:CompatibilityConditionforG}
0 < -G'(t) <  G(t),\quad\mbox{for all }t\geq 0.
\end{equation}
Clearly if $(u,b)$ is a solution to the IFPTK, we must have
\begin{equation}\label{e:InitialCompatibilityConditionforG}
	G(0)=\int^\infty_{-\infty} u_0(x) \;dx=1,
	\end{equation}
and	by formal differentiation,
\begin{equation}\label{e:InitialCompatibilityConditionforG'}
	G'(t)=-\int_{-\infty}^{b(t)} u(t,x) \;dx,\quad\mbox{for all }t\geq 0.
	\end{equation}
This will be proven rigorously below in Lemma \ref{lem:CompatibilityCondition2}. The above discussion gives us the following compatibility conditions:
\begin{defin}[Compatibility Conditions]\label{a:Compatibility_Conditions}
	We say that $(G,u_0,b_0)$ is a compatible data if $G\in C^1(\R_+)$, $u_0\in H^2(\R)$ and $b_0\in \R$ satisfy the following properties:
	\begin{enumerate}[(i)]
		\item $u_0(x)>0$ for all $x\in \R$,
		\item $G$ satisfies the compatibility condition~\eqref{e:CompatibilityConditionforG},
		\item $G(0)$ satisfies the initial compatibility condition~\eqref{e:InitialCompatibilityConditionforG}, and
		\item $G'(0)$, $u(0,x):=u_0(x)$ and $b(0)=b_0$ satisfy \eqref{e:InitialCompatibilityConditionforG'}.
	\end{enumerate}
\end{defin}
\begin{defin}[Weak Solutions]\label{def:WeakSoln}
Let $T>0$, $u\in C([0,T];H^1(\R))\cap L^2([0,T];H^2(\R))$ and $b\in C([0,T])$. We say that $(u,b)$ is a weak solution to the problem~\eqref{e:MainSystem}-\eqref{e:MassIdentity} if
\begin{equation}\label{e:DistributionalSoln}
\begin{aligned}
\int_0^T \int_{-\infty}^\infty u \partial_t \phi \;dxdt = \;& \int_{-\infty}^\infty u \phi |_{t=T} \;dx - \int_{-\infty}^\infty u_0 \phi |_{t=0} \;dx \\
& - \frac{1}{2} \int_0^T \int_{-\infty}^\infty u \partial_x^2 \phi \;dxdt + \int_0^T \int_{-\infty}^\infty \1_{(-\infty,b(t)]} u \phi \;dxdt
\end{aligned}
\end{equation}
holds for all $\phi\in C^\infty_c([0,T]\times\R)$, and
\[
G(t)=\int^\infty_{-\infty} u(x,t) \;dx,\quad\mbox{for all }t\geq 0.
\]
\end{defin}
The following is our main result.
\begin{restatable}{theorem}{main}\label{thm_exist&unique}
Let $(G,u_0,b_0)$ be a compatible data. Then for any fixed $T>0$, the system \eqref{e:MainSystem} has a unique weak solution $(u,b)$ on $[0,T]\times \R$ with $b\in C(\R_+)$ and $u\in C([0,T];H^1(\R))\cap L^2([0,T];H^2(\R))$ such that $u>0$ in $[0,T]\times\R$.

Furthermore, the solution $(u,b)$ satisfies
\begin{equation*}
u(t,x) = \E \left[ f(x+B_t) \exp\left(- \int_0^t \1_{(-\infty,0]}(x+B_{t-s}-b(s))\,ds\right)\right], ~~x\in \R, t\in[0,T]
\end{equation*}
and as a result
\begin{equation*}
\begin{split}
G(t) &= \int_\R \E \left[ f(x+B_t) \exp\left(- \int_0^t \1_{(-\infty,0]}(x+B_{t-s}-b(s))\,ds\right)\right]\,dx, ~~t\in[0,T]\\
&= \int_\R \E \left[\exp\left(-\lambda \int_0^t \1_{(-\infty,0]}(x+B_s-b(s))\,ds\right)\right] f(x)\,dx, ~~t\in[0,T],
\end{split}
\end{equation*}
where $f:=u_0$. This implies that the IFPTK has a unique continuous solution $b$.
\end{restatable}

\textbf{Structure of the Paper.} Section \ref{s:linearizedprob} is devoted to the study of a linearized version of our PDE. This linearized version is used in Section \ref{s:Scheme} to construct an approximation scheme which will be shown to converge to a weak solution $(u,b)$ of the constrained PDE system in Sections \ref{ss:convergence} and \ref{ss:consistency}. With the existence of a solution in hand we use in Section \ref{s:prob} a general version of the Feynman-Kac formula to prove that the weak solution gives the correct probabilistic interpretation. Making use of Feynman-Kac formula, we can prove some further properties of the weak solution $(u,b)$ which lead in Section \ref{ss:uniqueness} to the proof that the solution we constructed is actually the \textit{unique} solution to the constrained PDE system. We put all these pieces together in Section \ref{s:IFPTK} where we show that the constructed barrier $b$ is the unique solution to the Inverse First Passage Time Problem for Killed Brownian Motion (IFPTK). The solution to the First Passage Time Problem is given in Section \ref{s:first}. Applications to mathematical finance are showcased in Section \ref{s:app}. We finish by conjecturing in Section \ref{s:multi} a result for general one-dimensional diffusions.

\textbf{List of Notation.} The following is a list of spaces that we will use throughout the paper:
\begin{itemize}
\item For any non-negative integers $m,n$ and $T>0$ we define the space
\[
C^{m,n}([0,T]\times\R):= \left\{ f:[0,T]\times\R : f(x,\cdot)\in C^n([0,T]), x\in\R, ~f(\cdot,t)\in C^m(\R),~ t\in [0,T] \right\}
\]
	\item For any non-negative integer $s$, and non-empty subset $A\subset\R$, we define the space
	\[
		H^s(A) := \left\{ f:A\to\R : \|f\|_{H^s(A)} <\infty \right\}
	\]
	where the $H^s$ norm is given by
	\[
		\|f\|_{H^s(A)} := \left( \sum_{k=0}^s \int_A |\partial_x^k f (X)|^2 \;dx \right)^{1/2}.
	\]
	When $A:=\R$, we may lighten the notation by ignoring the $A$-dependence, namely, write $H^s:=H^s(\R)$ and $\|\cdot\|_{H^s} := \|\cdot\|_{H^s(\R)}$. In particular, when $s=0$, we will follow the standard notation to write $L^2:=H^0$ and $\|\cdot\|_{L^2} := \|\cdot\|_{H^0}$.
	\item For any positive constant $T$ and non-negative integer $s$, we define the spaces
	\[
		C([0,T];H^s(\R)) := \left\{ f:[0,T]\times\R\to\R : \sup_{0\leq t\leq T}\|f(t,\cdot)\|_{H^s} <\infty \right\}
	\]
	and
	\[
		L^2([0,T];H^s(\R)) := \left\{ f:[0,T]\times\R\to\R : \int_{0}^{T} \|f(t,\cdot)\|_{H^s}^2 \;dt <\infty \right\}.
	\]
\end{itemize}

\section{Linearized Problem}\label{s:linearizedprob}
In this section we will study the linearized problem, which will be one of the key ingredients used in Section \ref{s:existence} for solving the problem \eqref{e:MainSystem} under the constraint \eqref{e:MassIdentity}.

Let us begin with a simple inequality.
\begin{lemma} \label{inter}
For any real number $k>0$ and any measurable function $z:\R\to\R$, the following holds
\[
|z(x)|\leq \sqrt{2}\|z\|_{L^2((-\infty,x])}^\frac12 \|\pa_x z\|_{L^2((-\infty,x])}^\frac12\leq k \|\pa_x z\|_{L^2((-\infty,x])}+\frac{1}{2k}\|z\|_{L^2((-\infty,x])}.
\]
As a consequence
\[
 \|z\|_{L^\infty(\R)}\leq \sqrt{2}\|z\|_{L^2(\R)}^\frac12 \|\pa_x z\|_{L^2(\R)}^\frac12\leq k \|\pa_x z\|_{L^2(\R)}+\frac{1}{2k}\|z\|_{L^2(\R)}.
\]
Furthermore, for any $x$, $y\in \R$, one has
\[
|z(x)-z(y)|\le |x-y|^\frac{1}{2} \|\pa_x z\|_{L^2}.
\]
\end{lemma}
The proof of Lemma \ref{inter} follows immediately from the Agmon inequality, Young's inequality, and the Cauchy-Schwarz inequality, so the details will be omitted. In the following, we study a linearized problem.
\begin{claim}\label{cl:Existence&Uniqueness_of_Linear_Parabolic_Eqt}
Let $c:[0,\infty)\rightarrow \R$ be a given curve, and
\[
\psi(t,x):=\1_{(-\infty,c(t)]}(x),
\]
where $\1_{(-\infty,c(t)]}$ is the indicator function of the set $(-\infty,c(t)]$. Then for every $u_0\in H^1(\R)$, there exists a unique solution $u\in C([0,\infty);H^1(\R))\cap L^2([0,\infty);H^2(\R))$ to the Cauchy problem
\begin{equation}\label{e_cauchy}
\left\{
\begin{aligned}
\partial_t u-\frac{1}{2}\partial_x^2 u&=-\psi u,\\
u(0,x)&=u_0(x),\quad\mbox{for all } x\in \R.
\end{aligned}
\right.
\end{equation}
Furthermore, we also have the following estimate:
\bb\label{estLinftyH1&L2H2}
\sup_{0\leq t\leq T} \|u(t)\|_{H^1}+ \left(\int_0^T \|\pa_x u(t)\|_{H^1}^2 \;dt \right)^{\frac{1}{2}} \leq 3 \|u_0\|_{H^1}.
\ee
\end{claim}
\begin{proof} For any fixed time $T>0$, we define a map $L:C([0,T];H^1(\R))\to C([0,T];H^1(\R))$  by
$L(u):=v$ where $v$ is the unique $C([0,T];H^1(\R))$ solution to
\begin{align*}
\partial_t v-\frac{1}{2}\partial_{xx} v&=-\psi u,\\
v(0,x)&=u_0(x), \quad\mbox{for all }x\in \R.
\end{align*}
It follows from the standard theory for non-homogeneous heat equations that $L$ is well-defined.

The standard energy estimate for the heat equation yields
\[
\|L(u_1)-L(u_2)\|_{C([0,T];L^2)}\leq \frac{T}{1-T} \|u_1-u_2\|_{C([0,T];L^2)}.
\]
As a result, $L$ is a contraction mapping on $C([0,T];L^2(\R))$ provided that $T\in (0,1/2)$. Hence, by the contraction mapping principle, we can solve \eqref{e_cauchy} in $C([0,T];L^2(\R))$ uniquely within a short time $T$, say $T=1/4$. Note that this time $T=1/4$ does not depend on the initial condition $u_0$. The solvability in $C([0,\infty);L^2(\R))$ follows directly from the semi-group property for linear parabolic equations. Since one can obtain the $C([0,T];H^1(\R))$ and $L^2([0,T];H^2(\R))$ regularities by using the standard regularizing argument and estimate \eqref{estLinftyH1&L2H2}, it remains to verify \eqref{estLinftyH1&L2H2}.
\begin{rema}
The following steps have to be justified by arguments such as mollification or approximation by smooth functions. 
Technically, a weak solution to \eqref{e_cauchy} only satisfies the equation in the distributional sense, like \eqref{e:DistributionalSoln}. Therefore, in order to make the formal argument below rigorous, one can regularize the weak solution as follows: for any $\epsilon>0$ and $(y,s)\in\R\times[\epsilon,\infty)$, one may choose the test function $\phi(x,t) := \varphi_\epsilon(y-x,s-t):=\varphi((y-x)/\epsilon,(s-t)/\epsilon,)$ where $\varphi$ is a standard Friedrichs mollifier with a compact support in $[-1,1]^2$. As a result, the smooth function $u_\epsilon := u * \varphi_\epsilon$ satisfies the equation
\begin{equation}\label{e:Regularized&Linearized_PDE}
	\partial_t u_\epsilon-\frac{1}{2}\partial_x^2 u_\epsilon = - \psi u_\epsilon - \left\{ (\psi u)_\epsilon - \psi u_\epsilon \right\}
\end{equation}
classically in $\R\times[\epsilon,\infty)$, where the smooth function $(\psi u)_\epsilon:= (\psi u)*\varphi_\epsilon$. Now, one can apply the formal estimations below directly to~\eqref{e:Regularized&Linearized_PDE}, and obtain unifrom (in $\epsilon$) estimates with error terms that come from the commutator of multiplying $\phi$ and convoluting with the $\varphi_\epsilon$. These error terms can be shown to vanish as $\epsilon\to 0^+$ by using the standard properties of mollifiers and commutator estimates. For example, while deriving an analogue of \eqref{eq:enL2} for $u_\epsilon$, one will have an extra error term 
\[
	2 \int_{\R} u_\epsilon \left\{ (\psi u)_\epsilon - \psi u_\epsilon \right\} \;dx = O \left( \|u\|_{L^2} \left\{ \|(\psi u)_\epsilon - \psi u \|_{L^2} + \| \psi \|_{L^\infty} \| u - u_\epsilon \|_{L^2} \right\} \right),
\]
which converges to $0$ since $\phi u$ and $u\in C([\epsilon,\infty);L^2(\R))$. Hence, using the fact the $u_\epsilon\to u$ as $\epsilon\to 0^+$, one can obtain \eqref{eq:enL2} by passing to the limit in the corresponding estimate for $u_\epsilon$. One can also make other formal estimates below rigorous in the same manner. We will skip these technicalities since they are standard, tedious, and not the main idea of the proof. This types of standard arguments for mollification or approximation can be found in \cite{TaylorPDEIII}.
\end{rema}
Next, we are going to derive estimate \eqref{estLinftyH1&L2H2}. Multiplying equation \eqref{e_cauchy} by the solution $u$, and then integrating over $\R$, we have, by integrating by parts,
\bb
\label{eq:enL2}
\diff{}{t}\|u\|_{L^2}^2+\|\pa_x u\|_{L^2}^2+2\|u\|^2_{L^2((-\infty,c(t)])} =0.
\ee
This is the energy identity on the $L^2$ level.

Furthermore, we differentiate equation \eqref{e_cauchy} with respect to $x$, and obtain an equation for $w:=\pa_x u$ as follows:
\bb\label{e:diff}
\pa_t w-\frac{1}{2}w_{xx}=-\psi w +\delta_{c(t)} u,
\ee
where $\delta_{c(t)}$ is the Dirac delta function at the position $c(t)$.
We estimate $u(t,c(t))w(t,c(t))$ via Lemma \ref{inter}:
\bb\label{e:est}
\begin{split}
|u(t,c(t))w(t,c(t))|&\leq \frac{1}{2}|u(t,c(t))|^2+\frac{1}{2}|w(t,c(t))|^2\\
&\leq  \|u\|^2_{L^2((-\infty,c(t)])} + \|w\|_{L^2((-\infty,c(t)])}^2 + \frac{1}{3}\|\pa_x w\|_{L^2}^2.
\end{split}
\ee
Multiplying \eqref{e:diff} by $w$, and then integrating over $\R$, we obtain
\[
\diff{}{t}\|w\|_{L^2}^2+\|\partial_x w\|_{L^2}^2+2\|w\|_{L^2((-\infty,c(t)]}^2\leq 2 |u(t,c(t))w(t,c(t))|.
\]
We apply the estimate \eqref{e:est} to obtain
\[
\diff{}{t}\|w\|_{L^2}^2+\frac{1}{3}\|\partial_x w\|_{L^2}^2\leq 2 \|u\|^2_{L^2((-\infty,c(t)])}.
\]
The $L^2$ energy identity (\ref{eq:enL2}) implies that $2 \|u\|^2_{L^2((-\infty,c(t)])} = -\dfrac{d}{dt}\|u\|_{L^2}^2 - \|w\|_{L^2}^2$, and therefore, after integrating in time, we have
\[
\| u(t)\|_{H^1}^2+\frac{1}{3}\int_0^t \|\pa_x u\|_{H^1}^2\leq \| u(0)\|_{H^1}^2,
\]
which concludes the proof.
\end{proof}

We next investigate what happens to the solution of the equation above when the initial data is positive and integrable.
\begin{claim}\label{fd<u<fh}
\label{cl:posu0}
For every nonnegative $u_0$, which is not identically zero, there exists $f_d(t,x)$ which is strictly positive for $t>0$ such that every $C([0,T];H^1(\R))\cap L^2([0,T];H^2(\R))$ solution of \eqref{e_cauchy} satisfies, for all $0<t\leq T$,
\begin{equation}\label{e:u_geq_f_d}
 u(t,x)\geq f_d(t,x):=\frac{e^{-t}}{\sqrt{2\pi t}}\int_{\R} e^{-\frac{(x-y)^2}{2t}} u_0(y)dy.
\end{equation}
In addition, if the initial data $u_0$ is also integrable, then there exists a function $f_h(t,x)$, which can be written explicitly as well, such that $u(t,x)\leq f_h(t,x)$ and $\int_\R f_h(t,x)\,dx=\int_\R u_0\, dx$.
\end{claim}
\begin{proof}
As long as $u$ stays non-negative, it satisfies
\[
\partial_t u-\frac{1}{2}\partial_{xx} u+u\geq 0
\]
in the sense of distribution, and hence, by the comparison principle for linear parabolic equations,
\[
u\geq f_d,
\]
where the function $f_d$ can be written explicitly as
\[
f_d(t,x)=\frac{e^{-t}}{\sqrt{2\pi t}}\int_{\R} e^{-\frac{(x-y)^2}{2t}} u_0(y)dy,
\]
which is a solution of the initial value problem of damped heat equation
\[\left\{
\begin{aligned}
\partial_t f-\frac{1}{2}\partial_{xx}f+f&=0\\
f(x,0)&=u_0.
\end{aligned}
\right.\]
One can easily see that $f_d$ is strictly positive for all $t>0$.
Similarly, $u$ also satisfies
\[
\partial_t u-\frac{1}{2}\partial_{xx} u\leq  0,
\]
in the sense of distribution, so by the comparison principle for heat equations,
\[
u\leq f_h,
\]
where
\[
f_h(t,x)=\frac{1}{\sqrt{2\pi t}}\int_{\R} e^{-\frac{(x-y)^2}{2t}} u_0(y)dy.
\]
The function $f_h$ has the required property as it is a solution of the homogeneous heat equation.
\end{proof}

A direct consequence of Claim \ref{fd<u<fh} is the following
\begin{corr}
\label{pos}
Let $c\in C(\R^+)$ and $u$ be a $C([0,\infty);H^1(\R))\cap L^2([0,\infty);H^2(\R))$ solution to \eqref{e_cauchy} with a non-negative initial data $u_0$ satisfying $u_0(c(0))>0$. Then $u(c(t),t)>0$ for all $t\geq 0.$
\end{corr}

\section{Approximate Scheme}\label{s:Scheme}
In this section we will introduce and study the approximate scheme for the problem \eqref{e:MainSystem}-\eqref{e:MassIdentity}. More precisely, we will construct a sequence of approximate solutions by using an iteration scheme. The convergence of the approximate solutions will be studied in Section \ref{s:existence}.

First of all, let us introduce the iteration scheme as follows:
\begin{enumerate}[(i)]
	\item Define $u_1:=u_1(t,x)$ as the unique $C([0,T];H^1(\R))\cap C^\infty ((0,T]\times\R)$ solution to the homogeneous heat equation
	\begin{equation}\label{e:AppSolnInit}
	\left\{\begin{aligned}
	\partial_t u_1 -\frac{1}{2}\partial_x^2 u_1 &= 0\\
	u_1(0,x) &= u_0(x) >0.
	\end{aligned}\right.
	\end{equation}
	\item For any given $u_k$, we define the approximate barrier function $b_k :=b_k(t)$ by
	\begin{equation}\label{e:AppSolnDefnforb}
	-\int_{-\infty}^{b_k(t)} u_k (t,x) \;dx = G'(t).
	\end{equation}
	\item For any given $b_k$, we define $u_{k+1} := u_{k+1} (t,x) \in C([0,T];H^1(\R))\cap L^2 ([0,T];H^2(\R))$ as the unique weak solution to the parabolic equation
	\begin{equation}\label{e:AppSolnDefnforu}
	\left\{\begin{aligned}
	\partial_t u_{k+1} -\frac{1}{2}\partial_x^2 u_{k+1} &= - \1_{(-\infty,b_k(t)]} u_{k+1}\\
	u_{k+1}(0,x) &= u_0(x) >0.
	\end{aligned}\right.
	\end{equation}
\end{enumerate}
Regarding this iterative scheme, one may ask whether the sequence of approximate solutions can be defined iteratively for any given initial data $u_0>0$ and $G$ that satisfies the compatibility conditions \eqref{e:CompatibilityConditionforG} and \eqref{e:InitialCompatibilityConditionforG}. The answer is affirmative because of the following result.

\begin{prop}[Solvability of Approximate Scheme]\label{prop:SolvabilityofApproximateScheme}
For any fixed $T>0$, let $G\in C^1([0,T])$ and $u_0 \in H^2(\R)$ satisfy the compatibility conditions \eqref{e:CompatibilityConditionforG} and \eqref{e:InitialCompatibilityConditionforG}. Assume that $u_0>0$ on $\R$. Then there exists a unique sequence $\{(u_k,b_k)\}_{k=1}^\infty \in C([0,T];H^1(\R))\cap L^2 ([0,T];H^2(\R)) \times C([0,T])$ that satisfies \eqref{e:AppSolnInit}, \eqref{e:AppSolnDefnforb} and \eqref{e:AppSolnDefnforu} in the weak sense.

Furthermore, the estimate~\eqref{estLinftyH1&L2H2} holds for all solutions $u_k$, and the sequence $\{(u_k,b_k)\}_{k=1}^\infty$ has the following monotonicity property: for any integer $k\ge 1$,
\begin{equation}\label{e:MonotonicityforAppSoln}
\left\{\begin{aligned}
0 < f_d \leq u_{k+1} &\le u_k \\
b_{k+1} &\ge b_k,
\end{aligned}\right.
\end{equation}
where the positive function $f_d$ is defined in \eqref{e:u_geq_f_d}.
\end{prop}
\begin{proof}
	The existence proof is based on a monotonicity argument and mathematical induction.
	
	It follows directly from the standard theory for the homogeneous heat equation that
	\begin{enumerate}[(i)]
		\item we can always define a unique positive solution $u_1$ in $C([0,T];H^1(\R))\cap L^2 ([0,T];H^2(\R))$ via the explicit solution formula for the homogeneous heat equation;
		\item for all times $t\geq 0$,
		\begin{equation}\label{e:MassConservationforu_1}
		\int_{-\infty}^{\infty} u_1 (t,x) \;dx = \int_{-\infty}^{\infty} u_0 (x) \;dx.
    	\end{equation}
	\end{enumerate}
	According to the identity \eqref{e:MassConservationforu_1} and conditions \eqref{e:InitialCompatibilityConditionforG} and \eqref{e:CompatibilityConditionforG}, we have, for any $t\ge 0$,
	\[
	\int_{-\infty}^{\infty} u_1 (t,x) \;dx = \int_{-\infty}^{\infty} u_0 (x) \;dx = G(0) \ge G(t) > -G'(t),
	\]
	and hence, we can always define $b_1:=b_1 (t)$ via the identity \eqref{e:AppSolnDefnforb} by using the implicit function theorem.
	
	Given that we have already constructed $(u_k,b_k) \in C([0,T];H^1(\R))\cap L^2 ([0,T];H^2(\R)) \times C([0,T])$, we can construct $(u_{k+1},b_{k+1})$ as follows.
	
	First of all, according to Claim~\ref{cl:Existence&Uniqueness_of_Linear_Parabolic_Eqt}, we can always define $u_{k+1}$ via solving the initial value problem~\eqref{e:AppSolnDefnforu} provided that $b_k$ is known. In addition, $u_{k+1}\geq f_d > 0$ according to Claim~\ref{cl:posu0}.
	
	Now, using the fact that $b_k(t)\geq b_{k-1}(t)$ for all $t\in [0,T]$, we can show the monotonicity $u_{k+1}\leq u_k$ by the maximum principle. Formally, since $b_k(t)\geq b_{k-1}(t)$ and $u_{k+1}>0$,
	\[\begin{aligned}
	\left(\partial_t -\frac{1}{2} \partial_{xx} + \1_{(-\infty,b_{k-1}]} \right) (u_{k+1} - u_k) &= -\1_{(-\infty,b_k]} u_{k+1} + \1_{(-\infty,b_{k-1}]} u_{k+1} \\
	&=  - \1_{(b_{k-1},b_k]} u_{k+1} \leq 0.
	\end{aligned}\]
	As a result, if both $u_k$ and $u_{k+1}$ were classical solutions, one would be able to prove the monotonicity $u_{k+1}\leq u_k$ directly by using the standard maximum principle. However, both $u_k$ and $u_{k+1}$ are just weak solutions instead of classical solutions, so we have to modify the proof of the comparison principle by using the standard duality argument. The proof will be provided in Proposition~\ref{prop:Comparison_Principle} for readers' convenience.
	
	Having defined $u_{k+1}$, we can define $b_{k+1}$ as follows. For any fixed time $t\in [0,T]$ define $b_{k+1}(t)$ as the unique $\beta\in\R$ such that
	\[
	-\int_{-\infty}^{\beta} u_{k+1} (t,x) \;dx = G'(t).
	\]
	In other words, $\beta$ is a zero of the continuous function
	\[
	F(\alpha ,t) := G'(t) + \int_{-\infty}^\alpha u_{k+1} (t,x) \;dx.
	\]
	Since $u_{k+1}$ is positive, $F(\cdot ,t)$ is injective. Due to the compatibility condition~\eqref{e:CompatibilityConditionforG}, we know that
	\[
	\lim_{\alpha\to -\infty} F(\alpha ,t) = G'(t) < 0.
	\]
	Therefore, we can always find the unique zero of $F(\cdot ,t)$ by the implicit function theorem for continuous functions provided that
	\begin{equation}\label{e:limF>0_as_alpha_to_infty}
	\lim_{\alpha\to \infty} F(\alpha ,t) > 0.
	\end{equation}
	As a result, in order to show that $b_{k+1}$ is well-defined, it suffices to prove
	\begin{equation}\label{e:intu_k+1geqG}
	\int_{-\infty}^\infty u_{k+1} (t,x) \;dx \geq G(t)
	\end{equation}
	for all $t\in [0,T]$ because \eqref{e:CompatibilityConditionforG} and \eqref{e:intu_k+1geqG} imply \eqref{e:limF>0_as_alpha_to_infty}.
	
	Formally, one can show \eqref{e:intu_k+1geqG} very easily by using the equation~$\eqref{e:AppSolnDefnforu}$ and the monotonicity $u_k \geq u_{k+1} > 0$. A direct computation yields
	\[\begin{aligned}
	\dfrac{d}{dt} \int_{-\infty}^{\infty} u_{k+1} (t,x) \;dx &= -\int_{-\infty}^{b_k (t)} u_{k+1} (t,x) \;dx \\
	& \geq -\int_{-\infty}^{b_k (t)} u_{k} (t,x) \;dx = G'(t),
	\end{aligned}\]
	where the last equality follows from the definition of $b_k$. Integration implies \eqref{e:intu_k+1geqG} because of the initial compatibility condition \eqref{e:InitialCompatibilityConditionforG}. However, this argument is not rigorous because $u_{k+1}$ is not a classical solution to \eqref{e:AppSolnDefnforu}. The rigorous way to show \eqref{e:intu_k+1geqG} is to justify the above argument by using the standard test function technique. More precisely, one can choose a sequence of test functions that approximate the indicator/characteristic function of $[0,t]\times\R$. Applying these test functions to the definition of weak solution, and passing to the limit appropriately, one will obtain an integral version of the above argument, and this justifies \eqref{e:intu_k+1geqG}. 
	
	Finally, we can also prove the monotonicity $b_{k+1} \geq b_k$ by using the monotonicity $u_k \geq u_{k+1} > 0 $. More precisely, since $u_k \geq u_{k+1} > 0 $, it follows from the definitions of $b_k$ and $b_{k+1}$ that
	\[
	\int_{-\infty}^{b_{k+1}(t)} u_{k+1}(t,x) \;dx = -G'(t) = \int_{-\infty}^{b_{k}(t)} u_{k}(t,x) \;dx \geq \int_{-\infty}^{b_{k}(t)} u_{k+1}(t,x) \;dx,
	\]
	and hence,
	\[
	b_{k+1} (t) \geq b_k (t) \quad\mbox{for all $t\in [0,T]$},
	\]
	since $u_{k+1} > 0 $. This completes the proof of Proposition~\ref{prop:SolvabilityofApproximateScheme}.
\end{proof}

\section{Existence and Uniqueness}\label{s:existence}
In this section we will first show the existence of solutions stated in Theorem \ref{thm_exist&unique} by proving the convergence of the iterative sequence that was constructed in Section \ref{s:Scheme}. The convergence of approximate solutions and consistency of the limit of the sequence will be shown in Sections~\ref{ss:convergence} and \ref{ss:consistency} respectively. We show in Section \ref{s:prob} that the solution of the PDE gives the correct probabilistic interpretation. Finally, we will prove uniqueness in Section~\ref{ss:uniqueness} and put all the pieces together in Section \ref{s:IFPTK}.

\subsection{Convergence of the iterative scheme}\label{ss:convergence}

In this section we will prove that the sequence $\{(u_k,b_k)\}_{k=1}^\infty$ of approximate solutions converges uniformly to the limit $(\tilde{u},\tilde{b})$.

According to Proposition~\ref{prop:SolvabilityofApproximateScheme}, we know that the sequence $\{ u_k \}_{k=1}^\infty$ of approximate solutions is uniformly bounded in $C([0,T];H^1(\R))\cap L^2 ([0,T];H^2(\R))$ because the estimate~\eqref{estLinftyH1&L2H2} holds for the approximate sequence $\{ u_k \}_{k=1}^\infty$ as well. In addition, by using the evolution equation~$\eqref{e:AppSolnDefnforu}$, we know that the sequence $\{ \partial_t u_k \}_{k=1}^\infty$ is also uniformly bounded in $L^2 ([0,T];L^2(\R))$.

As a result, by the Banach-Alaoglu theorem, there exist a subsequence $\{ u_{k_j} \}_{j=1}^\infty$ and a function $\tilde{u}\in L^\infty([0,T];H^1(\R))\cap L^2 ([0,T];H^2(\R))$ with $\partial_t\tilde{u}\in L^2 ([0,T];L^2(\R))$ such that
\[
\left\{
\begin{aligned}
u_{k_j} &\rightarrow \tilde{u} &&a.e. \\
u_{k_j} &\rightharpoonup \tilde{u} && \mbox{in $L^2 ([0,T];H^2(\R))$} \\
\partial_tu_{k_j} &\rightharpoonup \partial_t\tilde{u} && \mbox{in $L^2 ([0,T];L^2(\R))$},
\end{aligned}
\right.
\]
as $j\to\infty$. This implies that $\tilde{u}\in C([0,T];H^1(\R))$, and hence, that $\tilde u$ is continuous on $[0,T]\times\R$.

On the other hand, it follows from the monotonicity~\eqref{e:MonotonicityforAppSoln} that the whole sequence (instead of the subsequence) of continuous functions $\{ u_k \}_{k=1}^\infty$ actually converges pointwise in $[0,T]\times\R$ because they are uniformly bounded below by $f_d$ according to Claim~\ref{fd<u<fh}. Due to the uniqueness of the pointwise limit we have that for any $(t,x)\in [0,T]\times\R$,
\[
\tilde{u} (t,x) = \lim\limits_{k\to\infty} u_k (t,x).
\]
According to Dini's theorem, the above convergence is uniform on any compact subset of $[0,T]\times\R$ due to the continuity of $\tilde{u}$ and the monotonicity~\eqref{e:MonotonicityforAppSoln}.

Furthermore, using the monotonicity~\eqref{e:MonotonicityforAppSoln} and Claim~\ref{cl:posu0}, we have, for any positive integer $k$ and any $(t,x)\in [0,T]\times\R$,
\[
0<f_d(t,x)\leq\tilde{u} (t,x) \leq u_k (t,x)\leq f_h(t,x),
\]
where $f_h(t,x)\in L^1(\R)$ for all time $t\geq 0$. Since $\tilde{u}$ is a pointwise limit of $\{u_k\}_{k=1}^\infty$, using Lebesgue's dominated convergence theorem, we can pass to the limit $k\to\infty$ in \eqref{e:intu_k+1geqG}, and obtain
\begin{equation}\label{e:inttildeugeqG}
\int_{-\infty}^\infty \tilde{u} (t,x) \;dx \geq G(t).
\end{equation}
It follows from the positivity of $\tilde{u}$, inequality~\eqref{e:inttildeugeqG} and compatibility condition~\eqref{e:CompatibilityConditionforG} that we can always define a unique continuous function $\tilde{b}:=\tilde{b}(t)$ via
\[
\int_{-\infty}^{\tilde{b}(t)} \tilde{u}(t,x) \;dx = - G'(t)
\]
by using the implicit function theorem.

Using the fact that $0<\tilde{u}\leq u_k$, we have
\[
\int_{-\infty}^{\tilde{b}(t)}\tilde{u}(t,x) \;dx = -G'(t) = \int_{-\infty}^{b_{k}(t)} u_{k}(t,x) \;dx \geq \int_{-\infty}^{b_{k}(t)} \tilde{u}(t,x) \;dx,
\]
and hence,
\[
b_k\leq \tilde{b}.
\]
By the definitions of $b_k$ and $\tilde{b}$ and Lebesgue's dominated convergence theorem, we obtain
\[
\int_{b_k}^{\tilde{b}} \tilde{u} \; dx = \int_{-\infty}^{b_k} u_k - \tilde{u} \; dx \leq \int_{-\infty}^{\infty} u_k - \tilde{u} \; dx \to 0,
\]
as $k\to\infty$. Since $\tilde{u} > 0$, we have
\[
\lim_{k\to\infty} b_k(t) = \tilde{b}(t) , \quad\mbox{for all $t\in [0,T]$.}
\]
Since $\tilde{b}$ is continuous, the above convergence is actually uniform according to Dini's theorem.

\subsection{Consistency of the limit}\label{ss:consistency}

In this section we will show that the limit $(\tilde{u},\tilde{b})$, which was constructed in the subsection~\ref{ss:convergence}, is a weak solution to the problem~\eqref{e:MainSystem}-\eqref{e:MassIdentity} in the sense of Definition~\ref{def:WeakSoln}.

Note that for any $k\geq 2$, $(u_{k+1},b_k)$ is a weak solution to the approximate system~\eqref{e:AppSolnDefnforu}, and hence, we have, for any test function $\phi\in C^\infty_c ([0,T]\times\R)$,
\begin{equation}\label{e:DistributionalApproxSoln}
\begin{aligned}
\int_0^T \int_{-\infty}^\infty u_{k+1} \partial_t \phi \;dxdt = \;& \int_{-\infty}^\infty u_{k+1} \phi |_{t=T} \;dx - \int_{-\infty}^\infty u_0 \phi |_{t=0} \;dx \\
& - \frac{1}{2} \int_0^T \int_{-\infty}^\infty u_{k+1} \partial_x^2 \phi \;dxdt + \int_0^T \int_{-\infty}^\infty \1_{(-\infty,b_k(t)]} u_{k+1} \phi \;dxdt.
\end{aligned}
\end{equation}
Since both the sequences $\{u_k\}_{k=1}^\infty$ and $\{b_k\}_{k=1}^\infty$ converge uniformly on the support of $\phi$, we can pass to the limit $k\to\infty$ in \eqref{e:DistributionalApproxSoln}, and show that $(\tilde{u},\tilde{b})$ satisfies the integral identity \eqref{e:DistributionalSoln}. The convergence of the last term on the right hand side of \eqref{e:DistributionalApproxSoln} is guaranteed by Lebesgue's dominated convergence theorem and the fact that $0<f_d\leq u_k\leq f_h$ for all positive integer $k$, where $f_h$ is the $L^1$ function given in Claim~\ref{cl:posu0}.

We next show that the limit $\tilde{u}$ satisfies the mass identity \eqref{e:MassIdentity}. Since $|\1_{(-\infty,b_k(t)]} u_{k}|\leq f_h \in L^1(\R)$ for any time $t\in [0,T]$, we can apply Lebesgue's dominated convergence theorem to \eqref{e:AppSolnDefnforb} to obtain the following identity: for any $t\in [0,T]$,
\begin{equation}\label{e:Relation_of_tildeb&G'}
\int_{-\infty}^{\tilde{b}(t)} \tilde{u} (t,x) \;dx = - G'(t).
\end{equation}
Integrating \eqref{e:Relation_of_tildeb&G'} with respect to the time $t$, we also have, for any time $t\in [0,T]$,
\begin{equation}\label{e:Relation_of_tildeu&tildeb&G}
\int_{0}^{t} \int_{-\infty}^{\tilde{b}(s)} \tilde{u} (s,x) \;dxds = G(0)-G(t).
\end{equation}

On the other hand, we have already shown that $(\tilde{u},\tilde{b})$ satisfies the integral identity \eqref{e:DistributionalSoln} for any arbitrary test function $\phi\in C^\infty_c ([0,T]\times\R)$. Applying \eqref{e:DistributionalSoln} with a sequence of test functions that approximate the indicator/characteristic function of $[0,t]\times\R$, one can prove that
\begin{equation}\label{e:-int_-infty^inftytildeudx}
\begin{aligned}
-\int_{-\infty}^{\infty} \tilde{u} (t,x) \;dx &= -\int_{-\infty}^{\infty} u_0 (t,x) \;dx + \int_{0}^{t} \int_{-\infty}^{\tilde{b}(s)} \tilde{u} (s,x) \; dxds.
\end{aligned}
\end{equation}
The identity~\eqref{e:Relation_of_tildeu&tildeb&G} and the initial compatibility condition \eqref{e:InitialCompatibilityConditionforG} combined with \eqref{e:-int_-infty^inftytildeudx} imply that $\tilde{u}$ satisfies the mass identity~\eqref{e:MassIdentity}.

\subsection{Probabilistic Interpretation}\label{s:prob}

We have shown that for any $T>0$ there exists $u\in C([0,T];H^1(\R))\cap L^2([0,T];H^2(\R))$ and $b\in C([0,T])$ such that
\begin{equation}\label{e:sol}
\begin{split}
\partial_t u(t,x) &= \frac{1}{2} \partial_{xx} u(t,x) - u(t,x)\1_{(-\infty,0]} (x-b(t)), \\
u(0,x)&= f(x)>0, \quad\mbox{for all } x\in \R,\\
\int_\R u(t,x)\,dx &= G(t), \quad\mbox{for all } t\in [0,T].
\end{split}
\end{equation}
in the weak sense of Definition \ref{def:WeakSoln}.

We want to show that $u(t,x)$ and $b(t)$ give the correct probabilistic interpretation. Our solutions do not have enough regularity and the killing rate is discontinuous and as such we cannot use the classical Feynman-Kac formula. Instead, we will make use of the recent result \cite[Theorem 3.4]{G16}.

\begin{theorem}\label{t:prob}
Suppose that for some $T>0$, we have $u\in C([0,T];H^1(\R))\cap L^2([0,T];H^2(\R))$, $b\in C([0,T])$ and \eqref{e:sol} is satisfied in the weak sense. Assume furthermore that $u_0=f\in L^2(\R)$ with $\int_\R f(y)\,dy=1$. Then for a.e. $x\in\R$,
\begin{equation}\label{e:FK}
u(t,x) = \E \left[ f(x+B_t) \exp\left(- \int_0^t \1_{(-\infty,0]}(x+B_{t-s}-b(s))\,ds\right)\right], ~~t\in[0,T]
\end{equation}
and as a result,
\begin{equation}\label{e:IFPT}
G(t) = \int_\R \E \left[\exp\left(- \int_0^t \1_{(-\infty,0]}(x+B_{s}-b(s))\,ds\right)\right] f(x) \,dx, ~~t\in[0,T]
\end{equation}

\end{theorem}
\begin{proof}
It is worth noting that our solutions are more regular than the requirement stated in \cite{G16}, for example $C([0,T];H^1(\R))\cap L^2([0,T];H^2(\R))\subset W^1((0,T; H^1(\R),L^2(\R))$.
In the notation from \cite{G16} we have
\begin{itemize}
\item Weight $\eta=0$.
\item Symbol $A(\xi)=\frac{1}{2} \xi^2$.
\item Killing rate $\kappa(t,x) = 1_{(-\infty,0]} (x-b(t))$.
\end{itemize}
It is easy to check that the symbol $A$ satisfies \cite[Conditions 3.2]{G16}. In particular, conditions (A1) and (A4) are trivial, and conditions (A2) and (A3) hold for $\alpha=2$. Furthermore, the interior of the set $\{x\in\R ~|~\Pr\{|B_t-x|<\varepsilon\}>0~\text{for all}~\varepsilon>0\}$ is clearly $\R$. As a result, in the notation of \cite{G16}, we have $\text{supp}(B_t)=\R$, for all $t>0. $

Applying \cite[Theorem 3.4]{G16} yields that for a.e. $x\in\R$,
\[
u(t,x) = \E \left[ f(x+B_t) \exp\left(- \int_0^t \1_{(-\infty,0]}(x+B_{t-s}-b(s))\,ds\right)\right], ~~t\in[0,T].
\]
By time-reversal for the Brownian motion $(B_t)_{t\geq 0}$, we obtain
\begin{equation*}
\begin{split}
G(t) &= \int_\R \E \left[ f(x+B_t) \exp\left(- \int_0^t \1_{(-\infty,0]}(x+B_{t-s}-b(s))\,ds\right)\right] \,dx\\
&= \E\int_\R  \left[ f(x+B_t) \exp\left(- \int_0^t \1_{(-\infty,0]}(x+B_{t-s}-b(s))\,ds\right)\right] \,dx\\
&= \int_\R \E \left[\exp\left(- \int_0^t \1_{(-\infty,0]}(x+B_{s}-b(s))\,ds\right)\right] f(x) \,dx ,
\end{split}
\end{equation*}
which completes the proof.

\end{proof}
The following result shows that any weak solution of the problem~\eqref{e:MainSystem}-\eqref{e:MassIdentity}  needs to satisfy the compatibility conditions \eqref{e:CompatibilityConditionforG}-\eqref{e:InitialCompatibilityConditionforG'}.
\begin{lemma}\label{lem:CompatibilityCondition2}
Suppose that for some $T>0$, we have $b\in C([0,T])$, $u\in W^1((0,T; H^1(\R),L^2(\R))$ and \eqref{e:sol} is satisfied in the weak sense. Assume furthermore that $u_0=f\in L^2(\R)$ with $\int_\R f(y)\,dy=1$ and $f>0$.
For any $t\in [0,T]$, we have
\begin{equation}\label{e:comp1}
G'(t) = -\int_{-\infty}^{b(t)} u(t,x)\,dx
\end{equation}
and
\begin{equation}\label{e:comp2}
0 < -G'(t) < G(t).
\end{equation}
\end{lemma}
\begin{proof}
By Theorem \ref{t:prob}, we have, for any $t\in [0,T]$,
\begin{equation*}
\begin{split}
G(t) &= \int_\R \E \left[ f(x+B_t) \exp\left(- \int_0^t \1_{(-\infty,0]}(x+B_{t-s}-b(s))\,ds\right)\right] \,dx\\
&= \int_\R \E \left[\exp\left(- \int_0^t \1_{(-\infty,0]}(x+B_{s}-b(s))\,ds\right)\right] f(x) \,dx.
\end{split}
\end{equation*}
Using the dominated convergence theorem, the time-reversal property for Brownian motion (see \cite{HP86}), and equation \eqref{e:FK}, we have
\begin{equation*}
\begin{split}
G'(t) &= - \E \int_\R \1_{(-\infty,0]}(x+B_{t}-b(t)) \left[\exp\left(- \int_0^t \1_{(-\infty,0]}(x+B_{s}-b(s))\,ds\right)\right] f(x) \,dx\\
&= - \E \int_\R \1_{(-\infty,0]}(x-b(t)) \left[\exp\left(- \int_0^t \1_{(-\infty,0]}(x+B_{t-s}-b(s))\,ds\right)\right] f(x+B_t) \,dx\\
&= - \int_{-\infty}^{b(t)} \E  \left[\exp\left(- \int_0^t \1_{(-\infty,0]}(x+B_{t-s}-b(s))\,ds\right)\right] f(x+B_t) \,dx\\
&= - \int_{-\infty}^{b(t)} u(t,x)\,dx.
\end{split}
\end{equation*}
Note that $u>0$ and $b<\infty$ imply that for all $t\in [0,T]$,
\[
0< -G'(t)=\int_0^{b(t)}u(t,x),dx<\int_0^{\infty}u(t,x),dx = G(t).
\]

\end{proof}

\subsection{Uniqueness}\label{ss:uniqueness}

In this subsection we will prove the uniqueness of the weak solution to \eqref{e:sol} by deriving an $L^1$ estimate for the difference between any two solutions $(u_1,b_1)$ and $(u_2,b_2)$. Technically, we will apply the doubling variables argument to derive the comparison between two weak solutions; see \eqref{e:L^1_Est_for_tilde_u} below. This method is a standard way to estimate the difference between weak solutions; see Kru\v{z}kov's famous work \cite{Kruzkov1970} on scalar conservation laws in several spatial dimensions for instance. The main observation for this proof is identity~\eqref{e:technical_identity_for_a_L^1_norm}, which, in a certain sense, allows us to ``convert'' the difference between $b_1$ and $b_2$ to the difference between $u_1$ and $u_2$. The details of the proof will be provided as follows.

Let $(u_1,b_1)$ and $(u_2,b_2)$ be two weak solutions to the problem~\eqref{e:MainSystem}-\eqref{e:MassIdentity} in the sense of Definition~\ref{def:WeakSoln} with the same initial data $u_0$, and $u_2>0$. For any test function $\psi\in C^\infty_c((0,T)\times\R)$, there exists $\epsilon_0>0$ such that $\supp\;\psi\subset (\epsilon_0,T)\times\R$. For any real numbers $k\in\R$ and $\epsilon\in (0,\epsilon_0)$, we can define $[\sign (u_1 - k) ]_\epsilon := [\sign (u_1 - k) ]*\varphi_\epsilon$, where $\varphi_\epsilon(t,x):=\epsilon^{-1}\varphi(t/\epsilon,x/\epsilon)$ and $\varphi$ is a standard mollifier supported in $(-1,1)^2$. Using $\phi := [\sign (u_1 - k) ]_\epsilon \psi$ as a test function in \eqref{e:DistributionalSoln}, we have 
\begin{equation}\label{e:Integral_Identity_for_u_1-k_with_epsilon}
	\begin{aligned}
	& \int_0^T \int_{-\infty}^\infty ( u_1 - k ) [\sign (u_1 - k) ]_\epsilon \partial_t \psi \;dxdt + \int_0^T \int_{-\infty}^\infty ( u_1 - k ) \psi \partial_t [\sign (u_1 - k) ]_\epsilon \;dxdt \\
	=\;& - \frac{1}{2} \int_0^T \int_{-\infty}^\infty \partial_x^2 ( u_1 - k ) [\sign (u_1 - k) ]_\epsilon \psi \;dxdt \\
	& + \int_0^T \int_{-\infty}^\infty \1_{(-\infty,b_1(t)]} ( u_1 - k ) [\sign (u_1 - k) ]_\epsilon \psi \;dxdt \\
	&+ \int_0^T \int_{-\infty}^\infty \1_{(-\infty,b_1(t)]} k [\sign (u_1 - k) ]_\epsilon \psi \;dxdt,
	\end{aligned}
\end{equation}
since $(u_1,b_1)$ satisfies \eqref{e:DistributionalSoln}. As $\epsilon\to 0^+$,
$[\sign (u_1 - k) ]_\epsilon \to \sign (u_1-k)$ in $L^p_{loc}((0,T)\times\R)$ for all $1\leq p<\infty$. Using this fact and the $C([0,T];H^1(\R))\cap L^2([0,T];H^2(\R))$ regularity of $u_1$, one can pass to the limit $\epsilon\to 0^+$ in \eqref{e:Integral_Identity_for_u_1-k_with_epsilon}, and obtain
\begin{equation}\label{e:Integral_Identity_for_u_1-k}
	\begin{aligned}
	& \int_0^T \int_{-\infty}^\infty | u_1 - k | \partial_t \psi \;dxdt \\
	=\;& - \frac{1}{2} \int_0^T \int_{-\infty}^\infty \partial_x^2 ( u_1 - k ) \sign (u_1 - k)  \psi \;dxdt \\
	& + \int_0^T \int_{-\infty}^\infty \1_{(-\infty,b_1(t)]} | u_1 - k | \psi \;dxdt + \int_0^T \int_{-\infty}^\infty \1_{(-\infty,b_1(t)]} k \sign (u_1 - k) \psi \;dxdt .
	\end{aligned}
\end{equation}
	
For any fixed $(s,y)\in [0,T]\times\R$, we can consider $u_2(s,y)$ as a well-defined constant since $u_2$ is a continuous function on $[0,T]\times\R$. For any test function $\psi:=\psi(t,x,s,y)\in C^\infty_c((0,T)\times\R\times (0,T)\times\R)$, we can first apply $k:=u_2(s,y)$ and $\psi:=\psi(t,x,s,y)$ to \eqref{e:Integral_Identity_for_u_1-k}, and then integrate with respect to $(s,y)$ over $[0,T]\times\R$, to obtain
\begin{equation}\label{e:Integral_Identity_for_u_1-u_2}
	\begin{aligned}
	& \int_0^T \int_{-\infty}^\infty \int_0^T \int_{-\infty}^\infty | u_1(t,x) - u_2(s,y) | \partial_t \psi \;dxdtdyds \\
	=\; & - \frac{1}{2} \int_0^T \int_{-\infty}^\infty \int_0^T \int_{-\infty}^\infty \partial_x^2 ( u_1(t,x) - u_2(s,y) ) \sign (u_1(t,x) - u_2(s,y)) \psi \;dxdtdyds \\
	& + \int_0^T \int_{-\infty}^\infty \int_0^T \int_{-\infty}^\infty \1_{(-\infty,b_1(t)]} (x) | u_1(t,x) - u_2(s,y) | \psi \;dxdtdyds \\
	& + \int_0^T \int_{-\infty}^\infty \int_0^T \int_{-\infty}^\infty \1_{(-\infty,b_1(t)]} (x) u_2(s,y) \sign (u_1(t,x) - u_2(s,y)) \psi \;dxdtdyds.
	\end{aligned}
\end{equation}
\begin{rema}
	It is worth noting that Equation~\eqref{e:Integral_Identity_for_u_1-u_2} is just a consequence of the fact that $(u_1,b_1)$ is a weak solutions to the problem~\eqref{e:MainSystem}-\eqref{e:MassIdentity}. We have not used the fact that $(u_2,b_2)$ is also a weak solution to the problem~\eqref{e:MainSystem}-\eqref{e:MassIdentity} yet.
\end{rema}
	
Now, using the fact that $(u_2,b_2)$ is also a weak solution to the problem~\eqref{e:MainSystem}-\eqref{e:MassIdentity}, one can adapt the above argument, and show that
\begin{equation}\label{e:Integral_Identity_for_u_2-u_1}
	\begin{aligned}
	& \int_0^T \int_{-\infty}^\infty \int_0^T \int_{-\infty}^\infty | u_1(t,x) - u_2(s,y) | \partial_s \psi \;dxdtdyds \\
	=\; & - \frac{1}{2} \int_0^T \int_{-\infty}^\infty \int_0^T \int_{-\infty}^\infty \partial_y^2 ( u_1(t,x) - u_2(s,y) ) \sign (u_1(t,x) - u_2(s,y)) \psi \;dxdtdyds \\
	& - \int_0^T \int_{-\infty}^\infty \int_0^T \int_{-\infty}^\infty \1_{(-\infty,b_2(s)]} (y) u_2(s,y) \sign (u_1(t,x) - u_2(s,y)) \psi \;dxdtdyds.
	\end{aligned}
\end{equation}
Summing up \eqref{e:Integral_Identity_for_u_1-u_2} and \eqref{e:Integral_Identity_for_u_2-u_1}, we finally obtain
\begin{equation}\label{e:Est_for_u_1-u_2_in_doubling_varialbes}
	\begin{aligned}
	& \int_0^T \int_{-\infty}^\infty \int_0^T \int_{-\infty}^\infty | u_1(t,x) - u_2(s,y) | \left( \partial_t \psi + \partial_s \psi \right) \;dxdtdyds \\
	=\; & - \frac{1}{2} \int_0^T \int_{-\infty}^\infty \int_0^T \int_{-\infty}^\infty \left\{ \partial_x^2 u_1(t,x) - \partial_y^2 u_2(s,y) \right\} \sign (u_1(t,x) - u_2(s,y)) \psi \;dxdtdyds \\
	& + \int_0^T \int_{-\infty}^\infty \int_0^T \int_{-\infty}^\infty \1_{(-\infty,b_1(t)]} (x) | u_1(t,x) - u_2(s,y) | \psi \;dxdtdyds \\
	& + \int_0^T \int_{-\infty}^\infty \int_0^T \int_{-\infty}^\infty \left\{ \1_{(-\infty,b_1(t)]} (x) - \1_{(-\infty,b_2(s)]} (y) \right\} u_2(s,y) \sign (u_1(t,x) - u_2(s,y)) \psi \;dxdtdyds .
	\end{aligned}
\end{equation}

Now, for any test function $\rho:=\rho(\tau,\xi)\in C^\infty_c([0,T]\times\R)$ and $h>0$, we choose
\begin{equation}\label{e:Choice_of_psi}
	\psi(t,x,s,y) := \rho \left( \frac{t+s}{2} , \frac{x+y}{2}  \right) \sigma_h (t-s) \sigma_h (x-y),
\end{equation}
where $\sigma_h (\cdot) := \frac{1}{h} \sigma (\frac{\cdot}{h})$ is a sequence of positive and smooth functions approximating the Dirac delta mass at the origin, namely
\[
	\sigma\in C^\infty_c(\R), \quad \sigma\geq 0, \quad \int_{-\infty}^\infty \sigma (z) \;dz = 1, \quad\mbox{and}\quad \supp\;\sigma \subseteq [-1,1].
\]
A direct computation yields
\[
		(\partial_t + \partial_s) \psi (t,x,s,y) = \partial_\tau \rho \left( \frac{t+s}{2} , \frac{x+y}{2}  \right) \sigma_h (t-s) \sigma_h (x-y),
\]
so applying the test function $\psi$ defined in \eqref{e:Choice_of_psi} to \eqref{e:Est_for_u_1-u_2_in_doubling_varialbes}, and then passing to the limit $h\to 0^+$, we have
\begin{equation}\label{e:Est_for_u_1-u_2}
\begin{aligned}
& \int_0^T \int_{-\infty}^\infty | u_1(t,x) - u_2(t,x) | \partial_t \rho (t,x) \;dxdt \\
=\; & - \frac{1}{2} \int_0^T \int_{-\infty}^\infty \partial_x^2 \left\{ u_1(t,x) - u_2(t,x) \right\} \sign (u_1(t,x) - u_2(t,x)) \rho (t,x) \;dxdt \\
& + \int_0^T \int_{-\infty}^\infty \1_{(-\infty,b_1(t)]} (x) | u_1(t,x) - u_2(t,x) | \rho (t,x) \;dxdt \\
& + \int_0^T \int_{-\infty}^\infty \left\{ \1_{(-\infty,b_1(t)]} (x) - \1_{(-\infty,b_2(t)]} (x) \right\} u_2(t,x) \sign (u_1(t,x) - u_2(t,x)) \rho (t,x) \;dxdt .
\end{aligned}
\end{equation}

Let $\tau\in (0,T]$ be an arbitrary time. In order to obtain the $L^1$ control on $u_1 - u_2$ at the time $t=\tau$, we choose the following test function:
\begin{equation}\label{e:Choice_of_rho}
	\rho(t,x) := [\1_{(\epsilon,\tau-\epsilon)} (t)]_\epsilon \cdot \zeta_R (x).
\end{equation}
Here, $[\1_{(\epsilon,\tau-\epsilon)} (t)]_\epsilon := [\1_{(\epsilon,\tau-\epsilon)} (t)]*\sigma_\epsilon$ and $\zeta_R (x) := \zeta (\frac{x}{R})$, where $\sigma_\epsilon$ is the mollifier that approximating the Dirac delta mass, and the function $\zeta\in C^\infty_c (\R)$ satisfies
\[
	\zeta\geq 0, \quad \zeta\equiv 1 \mbox{ in $[-1,1]$}, \quad\mbox{and}\quad \supp \;\zeta \subseteq [-2,2].
\]
Substituting \eqref{e:Choice_of_rho} into \eqref{e:Est_for_u_1-u_2}, integrating by parts with respect to $x$ in the first integral on the right hand side of \eqref{e:Est_for_u_1-u_2}, passing to the limit $R\to\infty$ first, and then passing to the limit $\epsilon\to 0^+$, we finally obtain
\begin{equation}\label{e:L^1_Est_for_u_1-u_2}
\begin{aligned}
& \int_{-\infty}^\infty | u_1(\tau,x) - u_2(\tau,x) | \;dx  \\
\leq\; & - \int_0^{\tau} \int_{-\infty}^\infty \1_{(-\infty,b_1(t)]} (x) | u_1(t,x) - u_2(t,x) |  \;dxdt \\
& - \int_0^{\tau} \int_{-\infty}^\infty \left\{ \1_{(-\infty,b_1(t)]} (x) - \1_{(-\infty,b_2(t)]} (x) \right\} u_2(t,x) \sign (u_1(t,x) - u_2(t,x)) \;dxdt 
\end{aligned}
\end{equation}
since both $u_1$ and $u_2$ satisfy the same initial condition $\eqref{e:MainSystem}_2$.

Define $\tilde{u}:=u_1-u_2$. Then estimate~\eqref{e:L^1_Est_for_u_1-u_2} implies that for any $\tau\in (0,T]$,
\begin{equation}\label{e:L^1_Est_for_tilde_u}
\|\tilde{u}(\tau)\|_{L^1(\R)} 
\leq \int_{0}^{\tau} \|\left( \1_{(-\infty,b_1(t)]} - \1_{(-\infty,b_2(t)]} \right) u_2(t)\|_{L^1(\R)} \;dt.
\end{equation}
On the other hand, using \eqref{e:comp1}, we have
\begin{equation}\label{e:G'_Identity}
\int_{-\infty}^{b_1(t)}  u_1 (t,x) \;dx = - G'(t) = \int_{-\infty}^{b_2(t)}  u_2 (t,x) \;dx,
\end{equation}
and hence, by $u_2>0$,
\begin{equation}\label{e:technical_identity_for_a_L^1_norm}
\begin{aligned}
& \|\left( \1_{(-\infty,b_1(t)]} - \1_{(-\infty,b_2(t)]} \right) u_2 (t) \|_{L^1(\R)} \\
=\;& \int_{-\infty}^{\infty}  \sign (b_1(t)-b_2(t))\left( \1_{(-\infty,b_1(t)]}(x) - \1_{(-\infty,b_2(t)]}(x) \right) u_2 (t,x) \;dx\\
=\;& - \sign (b_1(t)-b_2(t)) \int_{-\infty}^{\infty}  \1_{(-\infty,b_1(t)]}(x) \tilde{u} (t,x) \;dx.
\end{aligned}
\end{equation}
Applying \eqref{e:technical_identity_for_a_L^1_norm} to \eqref{e:L^1_Est_for_tilde_u}, we finally obtain
\[
\|\tilde{u}(\tau)\|_{L^1(\R)} \leq \int_{0}^{\tau} \| \1_{(-\infty,b_1(t)]} \tilde{u}(t) \|_{L^1(\R)} \;dt \leq \int_{0}^{\tau} \|\tilde{u}(t)\|_{L^1(\R)} \;dt. 
\]
It follows from Gr\"{o}nwall's inequality that $\tilde{u}\equiv 0$, or equivalently $u_1\equiv u_2$. Using $u_1\equiv u_2$ and the strict positivity of  $u_2$, we can conclude from \eqref{e:G'_Identity} that $b_1 \equiv b_2$.

\subsection{The solution to the IFPTK}\label{s:IFPTK}
We are ready to put all the pieces together and prove our main theorem.
\main*
\begin{proof}
From Sections \ref{ss:convergence} and \ref{ss:consistency}, combined with Theorem \ref{t:prob} we obtain the existence of a continuous barrier solving the IFPTK. The uniqueness follows from Section \ref{ss:uniqueness}. By Theorem \ref{t:prob}, the probabilistic interpretation works. This completes the proof.
\end{proof}

\section{The First Passage Time Problem for Killed Brownian Motion}\label{s:first}
We briefly describe how one can use PDE to study the first passage time problem for killed Brownian motion. Suppose we are given an initial density of the starting point of the Brownian motion $f(x)$ for $x\in\R$, and a barrier function $b: \R_+\to\R$.
\begin{theorem}
Assume $b:\R_+\to\R$ is measurable, $f\in L^2(\R)$, $f>0$ and $\int_\R f(y)\,dy=1$. Then for $T>0$ the system
\begin{equation*}
\left\{\begin{aligned}
\partial_t u &= \frac{1}{2} \partial_x^2 u - \1_{(-\infty,b(t)]} u, \quad\mbox{for all }t\in[0,T] \\
u(0,x) &= f(x), \quad\mbox{for all } x\in\R
\end{aligned}\right.
\end{equation*}
has a unique weak solution with $u\in  W^1((0,T); H^1(\R),L^2(\R))$. Furhermore, one can compute the survival function of the stopping time $\tau$ as
\begin{equation}
\begin{split}
\Pr\{\tau>t\} &=  \int_\R \E \left[\exp\left(- \int_0^t \1_{(-\infty,0]}(x+B_{s}-b(s))\,ds\right)\right] f(x) \,dx\\
&= \int_\R u(t,x)\,dt,\quad\mbox{for all }t\in [0,T].
\end{split}
\end{equation}
\end{theorem}
\begin{proof}
The proof follows once again from \cite[Theorem 3.4]{G16}, and as such is omitted.
\end{proof}

\section{Applications to Mathematical Finance}\label{s:app}
We provide a brief overview as to how one can use our model in mathematical finance. The random time $\tau$ can be seen as the default time of one of the parties involved in a financial agreement. This idea is not new. For example, in \cite{HW01} the authors model the default time as the first time a Brownian motion hits a time-dependent barrier. In our setting the Brownian motion can be seen as a credit index process. When the Brownian motion $B_t$ is large, this corresponds to a time $t$ when the counterparty is in sound financial health. As such, the killing rate $\1_{(-\infty,b(t)]}(B_t)$ is $0$ and default is unlikely. However, when $B_t$ is low, the killing rate $\1_{(-\infty,b(t)]}(B_t)$ is at its maximum value $1$ and default is more probable.

We remark that one can follow the method proposed in \cite{DP10} to calibrate the default time distribution of $\tau$ using the rates of credit default swaps (CDS).

We follow the strategy of \cite{EEH14} in order to showcase how one can price claims in this setting. Suppose that the asset price $(X_t)_{t\geq 0}$ is given by a geometric Brownian motion
\begin{equation}\label{e:asset}
dX_t = \mu X_t\,dt+\sigma X_t \,dW_t
\end{equation}
where $(W_t)_{t\geq 0}$ is a standard Brownian motion. Just as before, the default time is modeled by the Brownian motion $(B_t)_{t\geq 0}$. We do not assume that $(W_t)_{t\geq 0}$ and $(B_t)_{t\geq 0}$ are independent. Instead, we suppose that the two Brownian motions are correlated with correlation $\rho\in [-1,1]$. Without loss of generality one can write
\[
B_t =\rho B'_t + \sqrt{1-\rho^2}B_t''
\]
and
\[
W_t = B'_t
\]
for two independent Brownian motions $(B'_t)_{t\geq 0}$ and $(B''_t)_{t\geq 0}$. Suppose one wants to price contingent claims with a fixed maturity $T>0$ and payoff of the form
\[
F(X_T)\1_{\{\tau>T\}}.
\]
An immediate computation yields
\[
\E^x\left[F(X_T)\1_{\{\tau>T\}}\right] = \E^x\left[F(X_T)\exp\left(-\int_0^T \1_{(-\infty,0]}(B_s-b(s))\,ds\right)\right].
\]
In general, one is interested in the expected value of the payoff, given the past of the asset price and given that default did not happen yet. Therefore, one wants to be able to compute
\begin{equation*}
\begin{split}
\E^x\left[F(X_T)\1_{\{\tau>T\}}~|~(X_s)_{0\leq s\leq t}, \tau>t\right],
\end{split}
\end{equation*}
or equivalently
\[
\E^x\left[F(X_T)\exp\left(-\int_t^T \1_{(-\infty,0]}(B_s-b(s))\,ds\right)~\Bigg|~(X_s)_{0\leq s\leq t}, \tau>t\right].
\]

Consider the Markov process $Z_t:= (X_t,B_t)$. One can see that its generator acts on smooth enough functions $f$ via
\[
\mathcal{L} f= \frac{1}{2}\partial^2_{xx}f+\mu x \partial_xf +  \frac{1}{2}\partial^2_{yy}f + \rho\sigma\partial_x\partial_y f - 1_{(-\infty,0]}(y-b(t)).
\]
The Feynman-Kac formula from \cite{G16} tells us that if $F\in L^2(\R)$ and we set
\[
w(t,x,y):= \E^{(x,y)}\left[F(X_T)\exp\left(-\int_t^T \1_{(-\infty,0]}(B_s-b(s))\,ds\right)\right],
\]
then $w$ is the unique solution to
\begin{equation*}
\left\{\begin{aligned}
\partial_t w &=\mathcal{L}w,\quad\mbox{for all } t\in[0,T) \\
w(T,x,y) &= F(x), \quad\mbox{for all }x\in\R_+,\; y\in\R
\end{aligned}\right.
\end{equation*}
and $w\in W^1((0,T); H^1(\R),L^2(\R)).$ If we assume that the Brownian motion $(B_t)$ has a random starting point $B_0$ with density $f$, then we get
\[
\E^x\left[F(X_T)\exp\left(-\int_t^T \1_{(-\infty,0]}(B_s-b(s))\,ds\right)\right] = \int_\R w(t,x,y) f(y)\,dy.
\]

With this in hand we can follow the method from \cite[Section 5]{EEH14} to show that computing the price of a contingent claim in our setting reduces to solving certain PDE with coefficients depending on the path of the asset price.

\section{Killed Diffusions}\label{s:multi}

In this Section we will provide conjectures that generalize our results from a Brownian motion to general one-diumensional diffusions.

As before, suppose that $(B_t)_{t\geq 0}$ is a standard Brownian motion on a probability space $(\Omega,\F,\{\F_t\}_{t\geq0},\Pr)$ with a filtration $\{\F_t\}_{t\geq 0}$ satisfying the usual conditions. Define the one-dimensional diffusion $(Y_t)_{t\geq 0}$ via the SDE
\begin{equation}\label{e:Y}
dY_t = \mu(Y_t)\,dt + \sigma(Y_t)\,dB_t.
\end{equation}
We suppose that the functions $\sigma(\cdot)$ and $\mu(\cdot)\in C(\R)$ satisfy
\begin{itemize}
\item $\sigma(x)>0$ for all $x\in\R$; and
\item $\frac{1}{\sigma^2(\cdot)}$ and $\frac{\mu(\cdot)}{\sigma^2(\cdot)}$ are locally integrable on $\R$.
\end{itemize}
Under these conditions it is well known that the SDE \eqref{e:Y} has a solution that does not explode and is unique in law; see \cite{ES91} for instance. The process $(Y_t)_{t\geq 0}$ is a regular one-dimensional diffusion with scale function and speed measure densities given by
\begin{equation}\label{e:speed}
\begin{split}
m(dx) &= m(x)dx= \frac{2}{\sigma^2(x)}\exp\left(\int_0^x \frac{2}{\sigma^2(y)}\mu(y)\,dy\right)\\
S(dx) &= s(x)dx = \exp\left(-\int_0^x \frac{2}{\sigma^2(y)}\mu(y)\,dy\right).
\end{split}
\end{equation}
One can then define the random time
\bb\label{e_tau_Y}
\tau_Y:=\inf\left\{t>0: \lambda \int_0^t \1_{(-\infty,0]}(Y_s-b(s))\,ds>U\right\}
\ee
where $U$ is an independent exponential random variable with mean one. If $Y_0$ has a distribution with probability density $f$, then the lifetime of $\tau_Y$ can be computed as
\begin{equation}\label{e:survival_Y}
\Pr^x\{\tau_Y>t\}=\int_\R \E^x \left[\exp\left(- \int_0^t \1_{(-\infty,0]}(Y_{s}-b(s))\,ds\right)\right] f(x) \, dx.
\end{equation}

One can define the FPT and IFPT problems for the random time $\tau_Y$ from \eqref{e_tau_Y}. More specifically,
\begin{itemize}
\item The \textit{First Passage Time Problem for Killed Diffusions (FPTKD)}: For a given function $b:\R_+\to\R$, find the survival distribution of the first time that $(Y_t)_{t\geq 0}$ crosses $b$. That is, characterize
\[
\Pr\{ \tau_Y>t\},\quad\mbox{for all } t\geq 0.
\]
\item The \textit{Inverse First Passage Time Problem for Killed Diffusions (IFPTKD)}: For a given survival function $G$ on $(0,\infty)$ does there exist a function $b$ such that (using \eqref{e_survival_tau_2})
\[
G(t) = \Pr\{ \tau_Y>t\} = \int_\R \E^x \left[\exp\left(-\int_0^t \1_{(-\infty,0]}(Y_s-b(s))\,ds\right)\right] f(x)\,dx,
\]
for all $t\geq 0$?
\end{itemize}
Any one-dimensional diffusion is time-reversible with respect to its
speed measure, $m$.  That is, if $Y_0$ is distributed as $m$, then $(Y_s)_{0
\le s \le t}$ has the same distribution as $(Y_{t-s})_{0 \le s \le t}$ (see \cite{HP86} for more general results). Using this we can rewrite \eqref{e:survival_Y} as

\begin{equation}\label{e:surv}
\begin{split}
\Pr^x\{\tau_Y>t\}&= \int_\R \E^x \left[\exp\left(- \int_0^t \1_{(-\infty,0]}(Y_{s}-b(s))\,ds\right)\right] f(x) \, dx\\
&= \int_\R \E^x \left[\exp\left(- \int_0^t \1_{(-\infty,0]}(Y_{s}-b(s))\,ds\right)\right] \frac{f(x)}{m(x)} m(x) \, dx\\
&= \int_\R \E^x \left[\frac{f(Y_0)}{m(Y_0)} \exp\left(- \int_0^t \1_{(-\infty,0]}(Y_{s}-b(s))\,ds\right)\right]  m(x) \, dx\\
&= \int_\R \E^x \left[\frac{f(Y_t)}{m(Y_t)} \exp\left(- \int_0^t \1_{(-\infty,0]}(Y_{t-s}-b(s))\,ds\right)\right]  m(x) \, dx.
\end{split}
\end{equation}
If we set
\[
u(t,x):= \E^x \left[\frac{f(Y_t)}{m(Y_t)} \exp\left(- \int_0^t \1_{(-\infty,0]}(Y_{t-s}-b(s))\,ds\right)\right],
\]
then by Feynman-Kac formula in \cite{G16}, $u$ should satisfy
\begin{equation}\label{e:PDE_Y}
\begin{split}
\partial_t u &= \frac{1}{2}\sigma^2(x)\partial_{xx}^2u + \mu(x)\partial_x u -\1_{(-\infty,b(t)]}u\\
u(0,x) &= \frac{f(x)}{m(x)}, \quad\mbox{for all }x\in \R.
\end{split}
\end{equation}
The IFPT problem in this setting therefore reduces to studying the existence and uniqueness of solutions $(u,b)$ to
\begin{equation}\label{e:PDE_int}
\begin{split}
\partial_t u &= \frac{1}{2}\sigma^2(x)\partial_{xx}^2u + \mu(x)\partial_x u -\1_{(-\infty,b(t)]}u\\
u(0,x) &= \frac{f(x)}{m(x)},\quad\mbox{for all } x\in \R\\
G(t) &= \int_\R u(t,x) m(x)\,dx,\quad\mbox{for all } t\geq 0
\end{split}
\end{equation}
when the functions $G,f$ are given and $m$ is the speed measure from \eqref{e:speed}.

\begin{conj}
Assume that $G\in C^2(\R_+)$ is a function such that $G(0)=\int_\R f(x)\,dx=1$, $f$ is strictly positive on $\R$ and $f/m \in L^2(\R)$. Then for any fixed $T>0$, \eqref{e:PDE_int} has a unique weak solution $(u,b)$ on $[0,T]$ such that $u/m\in C([0,T];H^1(\R))\cap L^2([0,T];H^2(\R))$ and $b\in C([0,T])$. Furthermore, this implies that $b$ is the unique barrier such that
\[
G(t) = \Pr\{ \tau_Y>t\} = \int_\R \E^x \left[\exp\left(-\int_0^t \1_{(-\infty,0]}(Y_s-b(s))\,ds\right)\right] f(x)\,dx, t\in [0,T].
\]
\end{conj}

{\bf Acknowledgments.} The authors thank Kathrin Glau and Steve Evans for helpful discussions. T. K. Wong is partially supported by the HKU Seed Fund for Basic Research under the project code 201702159009, and the start-up Allowance for Croucher Award Recipients. A. Hening is supported by the NSF through the grant DMS 1853463.

\appendix

\section{Comparison Principle}\label{app:Comparison_Principle}

Below we prove a comparison principle for distributional solutions.

\begin{prop}\label{prop:Comparison_Principle}
	Let $(u^i,b^i)\in C([0,T];H^1(\R))\cap L^2 ([0,T];H^2(\R)) \times C([0,T])$ satisfy the integral identity~\eqref{e:DistributionalSoln} for any test function $\phi\in C^\infty_c([0,T]\times\R)$ with the same initial data $u_0$ in $\eqref{e:MainSystem}$, for $i=1$, $2$. Assume that $b^1\leq b^2$ in $[0,T]$ and $u^2\geq 0$ in $[0,T]\times\R$. Then
	\begin{equation}\label{e:Comparison_of_u}
	u^1 \geq u^2 \quad\mbox{in $[0,T]\times\R$.}
	\end{equation}
\end{prop}

\begin{proof}
Let $\tilde{u}:=u^1-u^2$. We will show that $\tilde{u}\geq 0$. For any arbitrary smooth and non-negative function $h$ with  compact support, we consider the following adjoint problem: for any $j=1$, $2$, $,\cdots$,
	\begin{equation}\label{e:AdjProb}
	\left\{\begin{aligned}
	- \partial_t \varphi_j - \frac{1}{2}\partial_x^2 \varphi_j &= - c_j \varphi_j + h\\
	\varphi_j (T,x) &= 0,
	\end{aligned}\right.
	\end{equation}
	where $\{c_j\}_{j=1}^\infty$ is a sequence of positive and smooth functions that approximate $\1_{(-\infty,b^1(t)]}$ in the following sense:
	\[
	c_j\to\1_{(-\infty,b(t)]} \quad\text{in } L^{\infty}([0,T]\times\R).
	\]
	Since all coefficients of the adjoint problem are smooth, it follows from the classical theory of scalar parabolic equations that there exists a unique classical solution $\varphi_j\in C^{1,2}([0,T]\times\R)$ to the adjoint problem~\eqref{e:AdjProb}. Furthermore, it follows from comparison principle and uniform $L^\infty$ bound of $\{c_j\}_{j=1}^\infty$ that the sequence $\{\varphi_j\}_{j=1}^\infty$ of unique solutions to the adjoint problem~\eqref{e:AdjProb} are indeed uniformly dominated by a function in $C([0,T];L^1(\R))$. In addition, using the facts that $h\geq 0$ and $c_j>0$, one may apply the classical maximum principle to show that $\varphi_j \geq0$ in $[0,T]\times\R$.
	
	Let $\phi\in C^\infty_c(\R)$ be a smooth cut-off function such that
	\[\left\{\begin{aligned}
	\phi &\equiv 1 &&\mbox{for all $-1\leq x\leq 1$}, &
	\phi &\equiv 0 &&\mbox{for all $|x|\geq 2$},\\
	|\phi'| &\leq 2 &&\mbox{for all $x\in\R$}, &
	|\phi''| &\leq 8 &&\mbox{for all $x\in\R$}.
	\end{aligned}\right.\]
	Define $\phi_R (x) := \phi(x/R)$. Then applying $\phi_R \varphi_j$ as test functions to the weak solutions $(u^1,b^1)$ and $(u^2,b^2)$, we have, via \eqref{e:DistributionalSoln},
	\begin{equation}\label{e:Weak_Soln_Identity_for_tildeu}
		\begin{aligned}
		&\int_0^T \int_{-\infty}^\infty \tilde{u} \partial_t (\phi_R \varphi_j) \;dxdt + \frac{1}{2} \int_0^T \int_{-\infty}^\infty \tilde{u} \partial_x^2 (\phi_R \varphi_j) \;dxdt \\
		&\qquad\qquad = \int_0^T \int_{-\infty}^\infty  \1_{(-\infty,b^1(t)]} \phi_R \varphi_j \tilde{u} \;dxdt - \int_0^T \int_{-\infty}^\infty \1_{(b^1(t),b^2(t)]} \phi_R \varphi_j u^2 \;dxdt.
		\end{aligned}
	\end{equation}
	Applying the adjoint equation~$\eqref{e:AdjProb}_1$ to \eqref{e:Weak_Soln_Identity_for_tildeu}, we obtain, since $\1_{(b^1(t),b^2(t)]}$, $\phi_R$, $\varphi_j$ and $u^2$ are all non-negative functions, that
	\begin{equation}\label{e:Weak_Soln_Ineq_wtih_I1_I2}
	\begin{aligned}
	0 &= \int_0^T \int_{-\infty}^\infty \phi_R h \tilde{u} \;dxdt - \int_0^T \int_{-\infty}^\infty \1_{(b^1(t),b^2(t)]} \phi_R \varphi_j u^2 \;dxdt + \mathcal{I}_1 + \mathcal{I}_2 \\
	&\leq \int_0^T \int_{-\infty}^\infty \phi_R h \tilde{u} \;dxdt + \mathcal{I}_1 + \mathcal{I}_2 ,
	\end{aligned}
	\end{equation}
 where
	\[\begin{aligned}
		\mathcal{I}_1 &:= \int_0^T \int_{-\infty}^\infty \left( \1_{(-\infty,b^1(t)]} - c_j \right) \phi_R \varphi_j \tilde{u} \;dxdt \\
		\mathcal{I}_2 &:= - \int_0^T \int_{-\infty}^\infty \tilde{u} \phi_R' \pa_x\varphi_j \;dxdt - \frac{1}{2} \int_0^T \int_{-\infty}^\infty \tilde{u} \phi_R'' \varphi_j \;dxdt \\
		&= \int_0^T \int_{-\infty}^\infty \left\{ \phi_R' \pa_x\tilde{u} + \frac{1}{2} \phi_R'' \tilde{u} \right\} \varphi_j \;dxdt .
	\end{aligned}
	\]
	It follows from Lebesgue's dominated convergence theorem that as $R\to\infty$,
	\[
		\mathcal{I}_1\to \int_0^T \int_{-\infty}^\infty \left( \1_{(-\infty,b^1(t)]} - c_j \right) \varphi_j \tilde{u} \;dxdt \quad\mbox{and}\quad \mathcal{I}_2\to 0,
	\]
	so passing to the limit $R\to\infty$ in \eqref{e:Weak_Soln_Ineq_wtih_I1_I2}, we have
	\begin{equation}\label{e:Weak_Soln_Ineq_after_R_to_infty}
		0 \leq \int_0^T \int_{-\infty}^\infty h \tilde{u} \;dxdt + \int_0^T \int_{-\infty}^\infty \left( \1_{(-\infty,b^1(t)]} - c_j \right) \varphi_j \tilde{u} \;dxdt.
	\end{equation}
	Passing to the limit $j\to\infty$ in \eqref{e:Weak_Soln_Ineq_after_R_to_infty}, by using Lebesgue's dominated convergence theorem again, we finally obtain
	\begin{equation}\label{e:Weak_Soln_Ineq}
	0 \leq \int_0^T \int_{-\infty}^\infty h \tilde{u} \;dxdt.
	\end{equation}
	Since Inequality~\eqref{e:Weak_Soln_Ineq} holds for any arbitrary smooth and non-negative function $h$ with compact support, we have $\tilde{u}\geq 0$. This completes the proof of Proposition~\ref{prop:Comparison_Principle}.
\end{proof}

\bibliographystyle{amsalpha}
\bibliography{IPT}
\end{document}